\documentclass[12pt]{article}
\usepackage{amsthm,amssymb,amsmath,enumerate}

\newtheorem{theorem}{Theorem}
\newtheorem{lemma}[theorem]{Lemma}
\newtheorem{claim}{Claim}

\title{\Large Approximating Connected Safe Sets in Weighted Trees}
\author{Stefan Ehard \and Dieter Rautenbach}
\date{}

\begin{document}

\maketitle

\begin{center}
Institut f\"{u}r Optimierung und Operations Research,\\ 
Universit\"{a}t Ulm, 89081 Ulm, Germany,\\
\{\texttt{stefan.ehard, dieter.rautenbach}\}\texttt{@uni-ulm.de}\\[3mm]
\end{center}

\begin{abstract}
For a graph $G$ and a non-negative integral weight function $w$ on the vertex set of $G$, a set $S$ of vertices of $G$ is $w$-safe if $w(C)\geq w(D)$ for every component $C$ of the subgraph of $G$ induced by $S$ and every component $D$ of the subgraph of $G$ induced by the complement of $S$ such that some vertex in $C$ is adjacent to some vertex of $D$. The minimum weight $w(S)$ of a $w$-safe set $S$ is the safe number $s(G,w)$ of the weighted graph $(G,w)$, and the minimum weight of a $w$-safe set that induces a connected subgraph of $G$ is its connected safe number $cs(G,w)$.
Bapat et al.~showed that computing $cs(G,w)$ is NP-hard even when $G$ is a star.
For a given weighted tree $(T,w)$, they described an efficient $2$-approximation algorithm for $cs(T,w)$ as well as an efficient $4$-approximation algorithm for $s(T,w)$.
Addressing a problem they posed, we present a PTAS for the connected safe number of a weighted tree. Our PTAS partly relies on an exact pseudopolynomial time algorithm, which also allows to derive an asymptotic FPTAS for restricted instances. 
Finally, we extend a bound due to Fujita et al.~from trees to block graphs.\\[5mm]
{\bf Keywords:} Safe set; connected safe set; block graph
\end{abstract}

\pagebreak

\section{Introduction}

In \cite{fumasa} Fujita, MacGillivray, and Sakuma define a {\it safe set} in a graph $G$ of order $n(G)$
as a set $S$ of vertices of $G$ with the property that 
$n(C)\geq n(D)$ for 
every component $C$ of the subgraph $G[S]$ of $G$ induced by $S$
and
every component $D$ of the subgraph $G-S$ of $G$ induced by $V(G)\setminus S$
such that some vertex in $C$ is adjacent to some vertex in $D$.
This notion can be motivated by different applications such as facility location \cite{fumasa}
or network majority \cite{bafulemamasatu}.
The smallest cardinality of a safe set in $G$ is the {\it safe number} $s(G)$ of $G$,
and the smallest cardinality of a safe set $S$ in $G$ that induces a connected subgraph in $G$ 
is the {\it connected safe number} $cs(G)$ of $G$.
For convenience, we call a set $S$ of vertices of $G$ {\it connected} if $G[S]$ is connected,
and we call two disjoint subgraphs $C$ and $D$ of $G$ 
{\it adjacent} if some vertex in $C$ is adjacent to some vertex in $D$.

Fujita et al.~\cite{fumasa} show 
\begin{eqnarray}
cs(G) & \leq & \min\left\{ \left\lceil \frac{n(G)}{2}\right\rceil, 2s(G)-1\right\}\mbox{ for every graph $G$, and}\label{e1}\\
cs(T) &\leq &\left\lceil \frac{n(T)}{3}\right\rceil\mbox{ for every tree $T$}.\label{e2}
\end{eqnarray}
They establish the hardness of the two parameters in general,
and describe a linear time algorithm computing a connected safe set of minimum order $cs(T)$ for a given tree $T$.
\'{A}gueda et al.~\cite{agcofulemamamonaotsatuxu} present an efficient algorithm based on dynamic programming
computing a safe set of minimum order $s(T)$ for a given tree $T$, 
and they also show the tractability of the safe number for graphs of bounded treewidth and interval graphs.

In \cite{bafulemamasatu,bafulemamasatu2} Bapat et al.~consider a weighted version of (connected) safe sets.
For a graph $G$, 
a weight function $w:V(G)\to \mathbb{Z}_{\geq 0}$, 
a set $U$ of vertices of $G$,
and a subgraph $H$ of $G$, let 
the {\it $w$-weight $w(U)$ of $U$} be $\sum\limits_{u\in U}w(u)$,
and let the {\it $w$-weight $w(H)$ of $H$} be $w(V(H))$.
A set $S$ of vertices of $G$ is a {\it $w$-safe set} in $G$ if
$w(C)\geq w(D)$ for 
every component $C$ of $G[S]$ 
and
every component $D$ of $G-S$
such that $D$ is adjacent to $C$.
The safe number $s(G,w)$ of the weighted graph $(G,w)$ is the minimum $w$-weight $w(S)$ of a $w$-safe set $S$ in $G$, and
the connected safe number $cs(G,w)$ of $(G,w)$ is the minimum $w$-weight $w(S)$ of a connected $w$-safe set $S$ in $G$.
Adapting an argument from \cite{fumasa},
it is not difficult to show $cs(G,w)\leq 2s(G,w)$.

In view of the efficient algorithms for trees mentioned above,
Bapat et al.~\cite{bafulemamasatu,bafulemamasatu2} show the surprising result that computing the safe number and the connected safe number are NP-hard problems already for weighted stars.
On the positive side, 
they describe an efficient algorithm computing the safe number of a weighted path,
an efficient $2$-approximation algorithm for the connected safe number of a weighted tree,
and  
an efficient $4$-approximation algorithm for the safe number of a weighted tree.
In the third problem mentioned in the conclusion of \cite{bafulemamasatu2},
they ask for better approximation algorithms for trees, and, 
in particular, for polynomial-time approximation schemes (PTAS).

As our first main contribution, we provide one such PTAS for the connected safe number of a weighted tree.

\begin{theorem}\label{ptas2}
Let $\epsilon$ be a positive real.

For a given tree $T$ of order $n$, and a given function $w:V(T)\to \mathbb{Z}_{\geq 0}$,
a connected $w$-safe set $S$ with $w(S)\leq (1+\epsilon)cs(T,w)$ 
can be determined in time $O\left(\frac{1}{\epsilon^4}n^{O\left(\frac{1}{\epsilon}\right)}\right)$.
\end{theorem}
For instances $(T,w)$ with $cs(T,w)<\frac{1}{\epsilon}$,
the integrality of the weights implies that a $(1+\epsilon)$-approximate solution 
to the problem of finding a connected $w$-safe set of minimum $w$-weight
must actually be an optimal solution.
Therefore, as part of the proof of Theorem \ref{ptas2}, we need to solve such instances optimally.
Note that, unlike \cite{bafulemamasatu,bafulemamasatu2}, 
we allow $0$ as a vertex weight, which means that sets of bounded $w$-weight
might still contain an unbounded number of vertices.

A natural approach to obtain a fully polynomial time approximation scheme (FPTAS) for the connected safe number of a weighted tree $(T,w)$ is to adapt Ibarra and Kim's \cite{ibki} famous FPTAS for the knapsack problem.
Their approach uses scaled and rounded versions of the weights,
which only approximate the original weights.
This led us to consider a slightly more general problem, 
where instances consist of a tree $T$, 
and two weight functions
$w^-:V(T)\to \mathbb{Z}_{\geq 0}$
and
$w^+:V(T)\to \mathbb{Z}_{\geq 0}$,
and where a set $S$ of vertices of $T$ is a {\it $(w^-,w^+)$-safe set in $T$} if
$w^-(C)\geq w^+(D)$ for every component $C$ of $T[S]$ and 
every component $D$ of $T-S$ 
such that $D$ is adjacent to $C$.
Inspired by the dynamic programming approach used in \cite{agcofulemamamonaotsatuxu},
we obtain the following.

\begin{theorem}\label{theorem2}
For a given tree $T$ of order $n$, and two given functions 
$w^-:V(T)\to \mathbb{Z}_{\geq 0}$ and
$w^+:V(T)\to \mathbb{Z}_{\geq 0}$
such that $w^-(S)\leq W$
for some connected $(w^-,w^+)$-safe set $S$ in $T$ and some $W\in \mathbb{N}$,
a connected $(w^-,w^+)$-safe set in $T$ of minimum $w^-$-weight 
can be determined in $O(nW^4)$ time. 
\end{theorem}
Unfortunately, essential arguments used by Ibarra and Kim do not carry over,
and implementing their approach we only obtain the following result.

\begin{theorem}\label{theorem3}
Let $M$ be a positive integer.

For a given triple $(T,w,\epsilon)$, where 
$\epsilon$ is a positive real,
$T$ is a tree of order $n$, and 
$w:V(T)\to \mathbb{Z}_{\geq 0}$ is a function
such that 
\begin{enumerate}[(i)]
\item $\epsilon\leq \min\left\{\frac{1}{3},\frac{1}{M}\right\}$, and
\item $\max\{ w(u):u\in V(T)\}\leq M \min\{ w(u):u\in V(T)\}$,
\end{enumerate}
a connected $w$-safe set $S$ in $T$ with 
$w(S)\leq (1+3\epsilon+2\epsilon^2)\cdot cs(T,w)+w_{\max}$ 
can be determined in $O\left(\frac{(1+3\epsilon+M)^4n^5}{\epsilon^8}\right)$ time. 
\end{theorem}
We conjecture that $cs(T,w)$ is still intractable for instances $(T,w)$ that satisfy (ii) for some fixed $M$,
and provide the following result supporting this belief.

\begin{theorem}\label{theoremhard}
It is NP-hard to determine $cs(T,w)$
for a given star $T$ of order $n$, 
and a given weight function $w:V(T)\to \mathbb{N}$ 
for which there is one vertex $u^*$ of $T$ such that 
$\frac{1}{n}w(T)\leq 4\min\Big\{ w(u):u\in V(T)\setminus \{ u^*\}\Big\}$.
\end{theorem}
Our proofs of Theorem \ref{ptas2} and Theorem \ref{theorem3} 
both rely on an efficient $2$-approximation algorithm for the connected safe number of a weighted tree. 
Bapat et al.~\cite{bafulemamasatu2} provide such an algorithm
but they require the weight function to be strictly positive, and some of their arguments seem not to work when vertices of weight $0$ are allowed.
Therefore, we include the following result.

\begin{theorem}\label{thm2approx}
For a given tree $T$ of order $n$, 
and a given function $w:V(T)\to \mathbb{Z}_{\geq 0}$,
a connected $w$-safe set $S$ with $w(S)\leq 2cs(T,w)$ 
can be determined in time $O(n^2)$.
\end{theorem}
Since $cs(G,w)\leq 2s(G,w)$ for every weighted graph $(G,w)$, 
our approximation results for the connected safe number
immediately imply approximation results for the safe number,
where the approximation guarantees double.

Our last contribution concerns an upper bound.
Even though (\ref{e1}) and (\ref{e2}) are not difficult to show, it is not easy to improve these bounds.
One reason might be that the typical conditions involving, for instance, bounds on vertex degrees or on the connectivity, 
have the same effect on the components of both relevant subgraphs $G[S]$ and $G-S$ of $G$.
We extend (\ref{e2}) as follows to block graphs,
that is, to graphs whose blocks are all complete.

\begin{theorem}\label{theorembound1}
If $G$ is a connected block graph 
of order $n$ whose clique number equals $\omega$, 
then $cs(G)\leq \max\{ \lceil n/3\rceil,\lceil \omega/2\rceil\}$.
\end{theorem}
The next section contains all proofs.

\section{Proofs}

For an integer $k$, let $[k]$ be the set of all positive integers at most $k$,
and let $[k]_0=\{ 0\}\cup [k]$.
We use the usual conventions $\min\emptyset=\infty$ and $\max\emptyset=-\infty$.

Since Theorem \ref{ptas2} relies on Theorem \ref{theorem2} and Theorem \ref{thm2approx}, we prove these results first.
As explained in the introduction, we proceed similarly as in \cite{agcofulemamamonaotsatuxu} for Theorem \ref{theorem2}.

\begin{proof}[Proof of Theorem \ref{theorem2}]
Let $T$ be a tree, and let 
$w^-:V(T)\to \mathbb{Z}_{\geq 0}$
and 
$w^+:V(T)\to \mathbb{Z}_{\geq 0}$
be two functions.
We root the tree $T$ in some vertex $r$ of $T$.
For every vertex $u$ of $T$, we fix a linear order $v_1,\ldots,v_k$ of its children,
and, for every subset $I$ of $[k]$, 
let $T_{(u,I)}$ be the subtree of $T$ rooted in $u$ 
that contains $u$ as well as, for every $i$ in $I$,
the vertex $v_i$ and all descendants of $v_i$ in $T$.
Let the positive integer $W$ be such that $w^-(S)\leq W$ for some connected $(w^-,w^+)$-safe set $S$ in $T$.

For every vertex $u$ of $T$ with $k$ children, and every subset $I$ of $[k]$, let 
$$p_{(u,I)}:\{ 0,1\}\times [W]_0\times [W]_0\to \mathbb{Z}_{\geq 0}\cup \{ \pm\infty\}$$ 
be such that, for every $s$ and $a$ in $[W]_0$
\begin{itemize}
\item $p_{(u,I)}(1,s,a)$ is the minimum $b$ 
in $\mathbb{Z}_{\geq 0}\cup \{ \pm\infty\}$
for which there is some set $S$ of vertices of $T_{(u,i)}$ such that
\begin{enumerate}[(i)]
\item $u\in S$, 
\item $w^-(S)=s$ and $T_{(u,i)}[S]$ is connected,
\item $w^-(C_u)=a$, where $C_u$ is the component of $T_{(u,i)}[S]$ that contains $u$, 
\item $w^-(C)\geq w^+(D)$ for 
every component $C$ of $T_{(u,i)}[S]$ that is distinct from $C_u$, 
and 
every component $D$ of $T_{(u,i)}-S$
that is adjacent to $C$, and
\item $w^+(D)\leq b$ for
every component $D$ of $T_{(u,i)}-S$
that is adjacent to $C_u$.
\end{enumerate}
If no set $S$ satisfies (i), (ii), (iii), and (iv),
then let $p_{(u,I)}(1,s,a)=\infty$. 
\item $p_{(u,I)}(0,s,a)$ is the maximum $b$ 
in $\mathbb{Z}_{\geq 0}\cup \{ \pm\infty\}$
for which there is some set $S$ of vertices of $T_{(u,i)}$ such that
\begin{enumerate}[(i)]
\item $u\not\in S$, 
\item $w^-(S)=s$ and $T_{(u,i)}[S]$ is connected,
\item $w^+(D_u)=a$, where $D_u$ is the component of $T_{(u,i)}-S$ that contains $u$, 
\item $w^-(C)\geq w^+(D)$ for 
every component $C$ of $T_{(u,i)}[S]$,
and 
every component $D$ of $T_{(u,i)}-S$ 
that is adjacent to $C$ and distinct from $D_u$, and
\item $w^-(C)\geq b$ for
every component $C$ of $T_{(u,i)}[S]$
that is adjacent to $D_u$.
\end{enumerate}
If no set $S$ satisfies (i), (ii), (iii), and (iv),
then let $p_{(u,I)}(0,s,a)=-\infty$. 
\end{itemize}
The following three claims contain recursive formulas for some $p_{(u,I)}$.

\begin{claim}\label{claim1}
Let $u$ be a vertex of $T$, and let $s,a\in [W]_0$.
\begin{enumerate}[(a)]
\item $p_{(u,\emptyset)}(1,s,a)=
\begin{cases}
-\infty & \mbox{, if $s=a=w^-(u)$, and}\\
\infty & \mbox{, otherwise.}
\end{cases}$
\item $p_{(u,\emptyset)}(0,s,a)=
\begin{cases}
\infty & \mbox{, if $s=0$ and $a=w^+(u)$, and}\\
-\infty & \mbox{, otherwise.}
\end{cases}$
\end{enumerate}
\end{claim}
\begin{proof}[Proof of Claim \ref{claim1}]
Since $T_{(u,\emptyset)}$ contains only the vertex $u$,
the stated equalities follow immediately from the definitions.
\end{proof}

\begin{claim}\label{claim2}
Let $u$ be a vertex of $T$ with $k$ children $v_1,\ldots,v_k$, let $i\in [k]$,
and let $v_i$ have $\ell$ children.
Let $s,a\in [W]_0$.
\begin{enumerate}[(a)]
\item 
$$p_{(u,\{ i\})}(1,s,a)=
\begin{cases}
w^+(T_{(v_i,[\ell])}) & \mbox{, if $s=a=w^-(u)$,}\\
p_{(v_i,[\ell])}(1,s-w^-(u),a-w^-(u))
& \mbox{, if $s=a>w^-(u)$,}\\
\infty & \mbox{, otherwise.}
\end{cases}$$
\item 
$$p_{(u,\{ i\})}(0,s,a)=
\begin{cases}
\max\Big\{ a'\in [W]_0:p_{(v_i,[\ell])}(1,s,a')\leq a'\Big\} 
& \mbox{, if $a=w^+(u)$,}\\
p_{(v_i,[\ell])}(0,s,a-w^+(u))
& \mbox{, if $a>w^+(u)$,}\\
-\infty & \mbox{, otherwise.}
\end{cases}$$
\end{enumerate}
\end{claim}
\begin{proof}[Proof of Claim \ref{claim2}]
First, we consider (a).
Let $S$, $b$, and $C_u$ be as in the definition of $p_{(u,\{ i\})}(1,s,a)$
such that $p_{(u,\{ i\})}(1,s,a)=b$, if $S$ exists.
If $s\not=a$, then $p_{(u,\{ i\})}(1,s,a)$ is $\infty$ by definition,
since $C_u$ must coincide with $T_{(u,\{ i\})}[S]$, whenever the latter is connected.
If $s=a=w^-(u)$, then $S$ contains only the vertex $u$, 
and all vertices of $T_{(u,\{ i\})}-u=T_{(v_i,[\ell])}$ 
belong to one unique component $D$ outside of $S$.
By definition, $b$ is smallest with $b\geq w^+(D)$, 
that is, $p_{(u,\{ i\})}(1,s,a)$ equals $w^+(D)=w^+(T_{(v_i,[\ell])})$.
If $s=a>w^-(u)$, then $S$ contains $u$ as well as $v_i$.
Removing $u$ from $T_{(u,\{ i\})}$ and $S$ yields a suitable set for
$p_{(v_i,[\ell])}(1,s-w^-(u),a-w^-(u))$,
and vice versa, which implies the stated equality.
 
Next, we consider (b).
Let $S$, $b$, and $D_u$ be as in the definition of $p_{(u,\{ i\})}(0,s,a)$,
such that $p_{(u,\{ i\})}(0,s,a)=b$, if $S$ exists.
If $a=w^+(u)$, then $D_u$ contains only $u$, and $v$ belongs to $S$.
The condition $p_{(v_i,[\ell])}(1,s,a')\leq a'$ ensures that (iii) holds, 
and the largest possible value for $a'$ corresponds to the maximum value of $b$.
If $a>w^+(u)$, then $D_u$ contains $u$ as well as $v_i$.
Removing $u$ from $T_{(u,\{ i\})}$ reduces the third entry $a$ by $w^+(u)$,
and we obtain the stated equality with $p_{(v_i,[\ell])}(0,s,a-w^+(u))$.
If $a<w^+(u)$, then no set $S$ with the necessary properties exists.
\end{proof}

\begin{claim}\label{claim3}
Let $u$ be a vertex of $T$ with $k$ children $v_1,\ldots,v_k$, let $i\in [k]$,
and let $v_i$ have $\ell$ children.
Let $s,a\in [W]_0$.
\begin{enumerate}[(a)]
\item 
\begin{eqnarray*}
p_{(u,[i])}(1,s,a) & = & 
\min\Big\{
\max\Big\{ p_{(u,[i-1])}(1,s',a'),p_{(u,\{ i\})}(1,s'',a'')\Big\}:\\
&& \,\,\,\,\,\,\,\,\,\,\,\,\,\,\,\,\,\, 
s',s'',a',a''\in [W]_0\mbox{ with }s+w^-(u)=s'+s''\\
&& \,\,\,\,\,\,\,\,\,\,\,\,\,\,\,\,\,\, 
\mbox{and }a+w^-(u)=a'+a''\Big\}.
\end{eqnarray*}
\item 
\begin{eqnarray*}
p_{(u,[i])}(0,s,a) & = & 
\max\Big\{
p_{(u,[i-1])}\Big(0,s,a-w^+(T_{(v_j,[\ell])})\Big),\\
&& \,\,\,\,\,\,\,\,\,\,\,\,\,\,\,\,\,\,\, 
p_{(u,[i-1])}\Big(0,s,a-w^+(V(T_{(u,[i-1])})\setminus \{ u\})\Big)\Big\}.
\end{eqnarray*}
\end{enumerate}
\end{claim}
\begin{proof}[Proof of Claim \ref{claim3}]
First, we consider (a).
Let $S$, $b$, and $C_u$ be as in the definition of $p_{(u,\{ i\})}(1,s,a)$
such that $p_{(u,[i])}(1,s,a)=b$, if $S$ exists.
Considering $S'=S\cap V(T_{(u,[i-1])})$ and $S''=S\cap V(T_{(u,\{ i\})})$
implies that 
$p_{(u,[i])}(1,s,a)$ is at least the expression on the right hand side. Conversely, optimal choices for 
$S'\subseteq V(T_{(u,[i-1])})$ and $S''\subseteq V(T_{(u,\{ i\})})$ for the right hand side, and the set $S=S'\cup S''$ imply that
$p_{(u,[i])}(1,s,a)$ is at most the expression on the right hand side.

Next, we consider (b).
Let $S$, $b$, and $D_u$ be as in the definition of $p_{(u,\{ i\})}(0,s,a)$,
such that $p_{(u,[i])}(0,s,a)=b$, if $S$ exists.
Since $S$ induces a connected subgraph of $T_{(u,[i])}$, and does not contain $u$,
we obtain that 
either $S\subseteq V(T_{(u,[i-1])})$
or $S\subseteq V(T_{(u,\{ i\})})$.
In the first case, the component $D_u$ contains all vertices in $V(T_{(v_j,[\ell])})$
and 
$p_{(u,[i])}(0,s,a)$ equals 
$p_{(u,[i-1])}\Big(0,s,a-w^+(T_{(v_j,[\ell])})\Big)$,
and,
in the second case, the component $D_u$ contains all vertices in 
$V(T_{(u,[i-1])})\setminus \{ u\}$
and 
$p_{(u,[i])}(0,s,a)$ equals 
$p_{(u,[i-1])}\Big(0,s,a-w^+(V(T_{(u,[i-1])})\setminus \{ u\})\Big)$.
Since $b$ is maximized, 
$p_{(u,[i])}(0,s,a)$ equals the larger of these two terms.
\end{proof}
Based on these claims,
it is now easy to complete the proof.
Indeed, it follows easily from the definition of $p_{(u,I)}$ 
that the minimum $w^-$-weight $w^-_{\rm opt}$ of a connected $(w^-,w^+)$-safe set in $T$ 
is the minimum value of $s\in [W]_0$ such that 
\begin{itemize}
\item either $p_{(r,[d_T(r)])}(1,s,a)\leq a$ for some $a\in [s]_0$
\item or $p_{(r,[d_T(r)])}(0,s,a)\geq a$ for some $a\in [s]_0$.
\end{itemize}
Therefore, computing, for every vertex $u$ of $T$ with $k$ children, 
all values $p_{(u,I)}(0/1,s,a)$, where 
\begin{enumerate}[(i)]
\item either $I=\emptyset$ and $s,a\in [W]_0$,
\item or $I=\{ i\}$ for some $i\in [k]$ and $s,a\in [W]_0$, 
\item or $I=[i]$ for some $i\in [k]$ and $s,a\in [W]_0$, 
\end{enumerate}
allows to determine $w^-_{\rm opt}$.
By Claim \ref{claim1}, 
each of the $O(nW^2)$ terms as in (i) can be computed in $O(1)$ time.
By Claim \ref{claim2}, 
each of the $O(nW^2)$ terms as in (ii) can be computed in $O(W)$ time.
By Claim \ref{claim3}, 
each of the $O(nW^2)$ terms as in (iii) can be computed in $O(W^2)$ time.
Therefore, 
memorizing optimal choices within the recursions, 
allows to determine $w^-_{\rm opt}$
as well as a connected $(w^-,w^+)$-safe set in $T$ of minimum $w^-$-weight
in $O(nW^4)$ time.
\end{proof}

\begin{proof}[Proof of Theorem \ref{thm2approx}]
Let $T$ and $w$ be as in the statement.
Let $w_{\max}=\max\{ w(u):u\in V(T)\}$.
Clearly, $w_{\max}\leq cs(T,w)\leq w(T)$.
Note that 
$cs(T,w)=0$ 
if and only if $w(T)=0$
if and only if the empty set is a connected $w$-safe set.
Therefore, we may assume that $cs(T,w)>0$.

We say that an integer $W$ in $[w(T)]_0$ is {\it nice} 
if there is a connected set $S$ of vertices of $T$ with 
$W\leq w(S)\leq W+w_{\max}$, and
$w(D)\leq W$ for every component $D$ of $T-S$.
In this case, the set $S$ {\it certifies} that $W$ is nice.
For a vertex $r$ of $T$,
we say that an integer $W$ in $[w(T)]_0$ is {\it $r$-nice} 
if there is a connected set $S$ of vertices of $T$ with 
$r\in S$,
$W\leq w(S)\leq W+w_{\max}$, and
$w(D)\leq W$ for every component $D$ of $T-S$.
Similarly as above, 
we say that such a set $S$ {\it certifies} that $W$ is $r$-nice.

If $S=\emptyset$ certifies that some integer $W$ is nice, 
then $T-S$ has only one component $T$, 
and $w(T)\leq W\leq w(S)=0$ implies the contradiction $w(T)=0$.
This implies that $W$ is nice if and only if 
$W$ is $r$-nice for some vertex $r$ of $T$,
and, hence, the smallest nice integer $W_{\min}$ 
equals the smallest integer that is $r$-nice for some vertex $r$ of $T$.

\setcounter{claim}{0}

Let $r$ be a vertex of $T$, and consider $T$ as being rooted in $r$.
For every vertex $u$ of $T$, let $c(u)=w(T_u)$, 
where $T_u$ is the subtree of $T$ containing $u$ as well as all descendants of $u$ in $T$ rooted in $r$.
The values $(c(u))_{u\in V(T)}$ can be determined in $O(n)$ time 
processing the vertices of $T$ in reverse breadth first search order.
Since $c(u)\geq c(v)$ whenever $u$ is a parent of $v$,
we can determine in $O(n)$ time an order $u(1),\ldots,u(n)$ 
of the vertices of $T$ such that 
\begin{itemize}
\item $r=u(1)$, 
\item $c(u(1))\geq \ldots \geq c(u(n))$, and 
\item $\{ u(1),\ldots,u(i)\}$ is connected for every $i\in [n]$.
\end{itemize}
Let 
$$I=\{ i\in [n-1]:c(u(i))>c(u(i+1))\}.$$
Let $1\leq i_1<\ldots <i_k\leq n-1$ be such that $I=\{ i_1,\ldots,i_k\}$.
Set 
$i_0=1$,
$i_{k+1}=n$, and
$c(u(i_{k+1}+1))=-\infty$.

Let 
$$W_r=\min\Big\{\max\Big\{c(u(i_j+1)),
w(u(1))+\cdots+w(u(i_j))-w_{\max}\Big\}:j\in [k+1]_0\Big\}.$$
Note that $W_r$ can be determined in $O(n)$ time.

\begin{claim}\label{claimapprox2}
$W_r$ is the smallest $r$-nice integer.
\end{claim}
\begin{proof}[Proof of Claim \ref{claimapprox2}]
If $j\in [k+1]_0$ is such that 
$$W_r=\max\left\{c(u(i_j+1)),
w(u(1))+\cdots+w(u(i_j))-w_{\max}\right\},$$
then iteratively adding to the set $\{ u(1),\ldots,u(i_j)\}$ 
vertices from $\{ u(i_j+1),\ldots,u(n)\}$
that have a neighbor in the current set
until the resulting set has $w$-weight at least $W_r$
yields a set $S_r$ that certifies that $W_r$ is $r$-nice.
This implies $W_{\min}\leq W_r$.
Conversely, if $j$ is the smallest index in $[k+1]_0$ such that 
$c(u(i_j+1))\leq W_{\min}$,
and $S_{\min}$ certifies that $W_{\min}$ is $r$-nice,
then $S_{\min}$ contains $\{ u(1),\ldots,u(i_j)\}$,
which implies 
$w(u(1))+\cdots+w(u(i_j))\leq w(S_{\min})\leq W_{\min}+w_{\max}$.
Altogether, we obtain 
$$W_{\min}\geq \max\left\{c(u(i_j+1)),
w(u(1))+\cdots+w(u(i_j))-w_{\max}\right\}\geq W_r.$$
\end{proof}
It follows that 
$W_{\min}=\min\{ W_r:r\in V(T)\}$ 
is the smallest nice integer.
Since each $W_r$ can be determined in $O(n)$ time,
$W_{\min}$ can be determined in $O(n^2)$ time.
Furthermore, a set $S_{\min}$ certifying that $W_{\min}$ is nice can be constructed in $O(n)$ time
similarly as the set $S_r$ in the proof of Claim \ref{claimapprox2}.
By definition, $cs(T,w)$ is nice, which implies $W_{\min}\leq cs(T,w)$.
Since $S_{\min}$ is a connected $w$-safe set, and
$$w(S_{\min})\leq W_{\min}+w_{\max}\leq cs(T,w)+cs(T,w)=2cs(T,w),$$
the proof is complete.
\end{proof}
In order to further prepare the proof of Theorem \ref{ptas2},
we need the following lemma.

\begin{lemma}\label{ptas1}
Let $T$ be a tree rooted in some vertex $r$, and let $w:V(T)\to \mathbb{Z}_{\geq 0}$.
Let $W$ and $\partial W$ be positive integers.
Let $S_0$ be the set of vertices of $T$ that consists of $r$ and all vertices $u$ of $T$ such that the subtree of $T$ that contains $u$ as well as all descendants of $u$ in $T$ has $w$-weight more than $W$.
Let $L$ be the set of vertices $u$ in $V(T)\setminus S_0$ with $w(u)>\partial W$.

There is a set $S$ of vertices of $T$ such that 
\begin{itemize}
\item $W\leq w(S)\leq W+\partial W$, 
\item $r\in S$, 
\item $T[S]$ is connected, and 
\item $w(D)\leq W$ for every component $D$ of $T-S$ 
\end{itemize}
if and only if 
there is a subset $L'$ of $L$ such that
\begin{enumerate}[(i)] 
\item $|L'|<\frac{W+\partial W}{\partial W}$,
\item the smallest subtree $T'$ of $T$ that contains $S_0\cup L'$ satisfies $w(T')\leq W+\partial W$ and $L'=L\cap V(T')$, and
\item the component $T''$ of $T-(L\setminus L')$ that contains $r$ satisfies $w(T'')\geq W$.
\end{enumerate}
\end{lemma}
\begin{proof}
First, suppose that the set $S$ has the stated properties. 
Since $w(S)\leq W+\partial W$, the set $S$ contains less than $\frac{W+\partial W}{\partial W}$ vertices $u$ with $w(u)> \partial W$.
This implies that the set $L'=L\cap S$ satisfies (i).
Since $S$ contains $r$, and $w(D)\leq W$ for every component $D$ of $T-S$, the set $S_0$ is a subset of $S$. This implies that, if $T'$ and $T''$ are as in the statement, then $T'$ is a subtree of $T[S]$, and $T[S]$ is a subtree of $T''$, which implies (ii) and (iii).
Altogether, the set $L'$ has the desired properties.

Next, suppose that some subset $L'$ of $L$ satisfies (i), (ii), and (iii).
Since $V(T')$ intersects $L$ exactly in $L'$, the tree $T'$ is a subtree of $T''$. Since $w(u)\leq \partial W$ for every vertex $u$ in $V(T'')\setminus V(T')$, the bounds on $w(T')$ and $w(T'')$ imply that iteratively adding to the tree $T'$ vertices from $V(T'')\setminus V(T')$ that have a neighbor in $T'$, one can construct a subtree of $T$ with vertex set $S$ such that $T'$ is a subtree of $T[S]$, the tree $T[S]$ is a subtree of $T''$, and $w(S)\leq W+\partial W$. Since $S_0$ is a subset of $S$, we obtain that $r\in S$ and that $w(D)\leq W$ for every component $D$ of $T-S$.
Altogether, the set $S$ has the desired properties.
\end{proof}

\begin{proof}[Proof of Theorem \ref{ptas2}]
Let $\epsilon$ be a fixed positive real.
Let $T$ be a tree of order $n$, let $w:V(T)\to\mathbb{Z}_{\geq 0}$, and let $W^*=cs(T,w)$.
We explain how to determine in time $O\left(\frac{1}{\epsilon^4}n^{\frac{4}{\epsilon}+O(1)}\right)$ 
a connected safe set in $T$ of $w$-weight at most $(1+3\epsilon)W^*$.
Using Theorem \ref{thm2approx}
we first determine in time $O(n^2)$ 
a connected safe set $S_1$ with 
$$W^*\leq w(S_1)\leq 2W^*.$$
Let $\partial W=\lfloor \epsilon w(S_1)/2\rfloor$. 

If $\partial W=0$, then $w(S_1)<2/\epsilon$, 
and an optimal connected safe set can be determined 
using Theorem \ref{theorem2} in time $O\left(\frac{n}{\epsilon^4}\right)$.
Therefore, we may assume that $\partial W$ is a positive integer, 
which implies 
$$\frac{\epsilon W^*}{4}\leq \frac{\epsilon w(S_1)}{4}\leq \partial W\leq \frac{\epsilon w(S_1)}{2}\leq \epsilon W^*.$$ 
Let $L$ be the set of vertices $u$ of $T$ with $w(u)>\partial W$.

We consider two cases.

\medskip 

\noindent {\bf Case 1} {\it There is no connected set $S$ of vertices of $T$ 
such that $W^*+\partial W\leq w(S)\leq W^*+2\partial W$,
and $w(D)\leq W^*$ for every component $D$ of $T-S$.}

\medskip 

\noindent Let $S^*$ be a connected safe set of minimum $w$-weight $W^*$.
Let $L^*=S^*\cap L$.
Let $T^*$ be the component of $T-(L\setminus L^*)$ that contains $S^*$.

If $$w(T^*)-w(S^*)=w(T^*)-W^*\geq \partial W,$$ 
then iteratively adding to $S^*$ vertices from $V(T^*)\setminus S^*$ 
that have a neighbor in $S^*$ yields a set $S$ having the stated properties. 
Since we assume that such a set does not exist, we obtain $w(T^*)-W^*<\partial W$, 
and $V(T^*)$ is a connected safe set of $T$ with 
$$w(T^*)\leq W^*+\partial W\leq (1+\epsilon)W^*.$$ 
Now, the key observation is that $T^*$ is completely determined by $L^*$ if this set in non-empty, 
and that $S^*$ contains at most $W^*/\partial W\leq 4/\epsilon$ many vertices from $L$. 
If $L^*$ is empty, then $T^*$ is one of the components of $T-L$. 
Therefore, considering the $O\left(n^{4/\epsilon}\right)$ different possible choices for $L^*$, a connected safe set of the form $V(T^*)$ that is of minimum $w$-weight can be determined in time $O\left(n^{\frac{4}{\epsilon}+O(1)}\right)$. 
In view of the above estimate of $w(T^*)$, this yields a $(1+\epsilon)$-approximation in this case.

\medskip 

\noindent {\bf Case 2} {\it There is a connected set $S$ of vertices of $T$ 
such that $W^*+\partial W\leq w(S)\leq W^*+2\partial W$,
and $w(D)\leq W^*$ for every component $D$ of $T-S$.}

\medskip 

\noindent For $i\in [k]$ with $k=\left\lceil\frac{4}{\epsilon}\right\rceil+1$, let
$W_i=(i-1)\partial W$. Note that $W_1=0$,  
\begin{eqnarray*}
W_k&\geq &\left(\left(\frac{4}{\epsilon}+1\right)-1\right)\partial W\geq \frac{4}{\epsilon}\cdot \frac{\epsilon w(S_1)}{4}=w(S_1),\mbox{ and}\\
W_k &\leq & \left(\frac{4}{\epsilon}+1\right)\partial W.
\end{eqnarray*}
For every $i\in [k]$, we determine whether there is a connected set $S_i$ of vertices of $T$ 
such that $W_i\leq w(S_i)\leq W_i+2\partial W$, and $w(D)\leq W_i$ for every component $D$ of $T-S_i$.
In view of Lemma \ref{ptas1}, for all $i\in [k]$, 
every time considering the $n$ choices for $r$,
this can be done in time $O\left(k\cdot n^{\frac{W_i+2\partial W}{2\partial W}+O(1)}\right)
=O\left(\frac{1}{\epsilon}n^{\frac{2}{\epsilon}+O(1)}\right)$.
Now, if $i\in [k-1]$ is such that $W_i\leq W^*\leq W_{i+1}=W_i+\partial W$, then the existence of the set $S$ implies that the set $S_{i+1}$ exists. Since $S_{i+1}$ is a connected $w$-safe set, and
$$w(S_{i+1})\leq W_{i+1}+2\partial W=W_i+3\partial W\leq W^*+3\partial W\leq (1+3\epsilon)W^*,$$ 
the set $S_{i+1}$ yields a $(1+3\epsilon)$-approximation in this case.

\medskip 

\noindent Altogether, returning the best of all the connected $w$-safe sets considered in Case 1 and Case 2 yields the desired $(1+3\epsilon)$-approximation, and,
estimating rather wastefully, requires $O\left(\frac{1}{\epsilon^4}n^{\frac{4}{\epsilon}+O(1)}\right)$ time,
which completes the proof. 
\end{proof}

\begin{proof}[Proof of Theorem \ref{theorem3}]
Let $M$ and $(T,w,\epsilon)$ be as in the statement,
that is, in particular, $w_{\max}\leq Mw_{\min}$,
where 
$w_{\max}=\max\{ w(u):u\in V(T)\}$ and 
$w_{\min}=\min\{ w(u):u\in V(T)\}$.
Using Theorem \ref{thm2approx}
we first determine in time $O(n^2)$ 
a connected safe set $S_1$ with 
$$cs(T,w)\leq w(S_1)\leq 2cs(T,w).$$
Let $t=\frac{\epsilon^2 w(S_1)}{n}$.

If $t<1$, then $w(S_1)\leq \frac{n}{\epsilon^2}$, 
and an optimal connected safe set can be determined 
using Theorem \ref{theorem2} in time $O\left(\frac{n^5}{\epsilon^8}\right)$.
Hence, we may assume that $t\geq 1$.
Let 
$$w^-:V(T)\to \mathbb{Z}_{\geq 0}:u\mapsto \lfloor w(u)/t\rfloor\mbox{, and }
w^+:V(T)\to \mathbb{Z}_{\geq 0}:u\mapsto \lceil w(u)/t\rceil.$$
Since $w_{\max}\geq \frac{w(S_1)}{n}=\frac{t}{\epsilon^2}$, and
$w_{\min}\geq \frac{w_{\max}}{M}\geq \frac{t}{\epsilon^2 M}\geq \frac{t}{\epsilon}$, 
we obtain
\begin{eqnarray*}
w^-(u) & \geq & \left(\frac{\frac{1}{\epsilon}-1}{\frac{1}{\epsilon}}\right)\frac{w(u)}{t}=(1-\epsilon)\frac{w(u)}{t},\mbox{ and}\\
w^+(u) & \leq & \left(\frac{\frac{1}{\epsilon}+1}{\frac{1}{\epsilon}}\right)\frac{w(u)}{t}=(1+\epsilon)\frac{w(u)}{t}.
\end{eqnarray*}
for every vertex $u$ of $T$.

Trivially, if a set of vertices of $T$ is $(w^-,w^+)$-safe, then it is also $w$-safe.
The following claim contains an approximate inverse of that statement,
which is a key observation, and motivated the crucial condition (ii).

\setcounter{claim}{0}

\begin{claim}\label{claim4}
If $S$ is a connected $w$-safe set in $T$, then there is a connected $(w^-,w^+)$-safe set $S'$ in $T$ with 
$w^-(S')\leq (1+3\epsilon)w^-(S)+w_{\max}/t$.
\end{claim}
\begin{proof}[Proof of Claim \ref{claim4}]
Let $D$ be a component of $T-S$ such that $w^+(D)$ is maximum. Since $w(S)\geq w(D)$, we obtain
\begin{eqnarray*}
w^+(D) & \leq & (1+\epsilon)\frac{w(D)}{t}
\leq (1+\epsilon)\frac{w(S)}{t}
\leq \left(\frac{1+\epsilon}{1-\epsilon}\right)w^-(S)\\
&\stackrel{\epsilon\leq 1/3}{\leq} &(1+3\epsilon)w^-(S).
\end{eqnarray*}
Therefore, iteratively adding to $S$ vertices from $V(T)\setminus S$ that have a neighbor in $S$
as long as necessary in order to obtain a connected $(w^-,w^+)$-safe set,
yields a connected $(w^-,w^+)$-safe set $S'$ with
$$w^-(S')\leq (1+3\epsilon)w^-(S)+\max\{ w^-(u):u\in V(T)\},$$
which implies the desired result.
\end{proof}
Applying Claim \ref{claim4} to the connected $w$-safe set $S_1$ 
yields a connected $(w^-,w^+)$-safe set $S_1'$ with 
\begin{eqnarray*}
w^-(S_1')
&\leq &(1+3\epsilon)w^-(S)+\frac{w_{\max}}{t}\\
&\leq &(1+3\epsilon)\frac{w(S_1)}{t}+\frac{Mw_{\min}}{t}\\
&\leq &(1+3\epsilon+M)\frac{w(S_1)}{t}\\
&=&(1+3\epsilon+M)\frac{n}{\epsilon^2}.
\end{eqnarray*}
Using Theorem \ref{theorem2},
a connected $(w^-,w^+)$-safe set $S_2$ in $T$ of minimum $w^-$-weight can be determined 
in time $O\left(\frac{(1+3\epsilon+M)^4n^5}{\epsilon^8}\right)$.
Let $S_{\rm opt}$ be a connected $w$-safe set in $T$ of minimum $w$-weight.
By Claim \ref{claim4},
there is a connected $(w^-,w^+)$-safe set $S_{\rm opt}'$ with 
$w^-(S_{\rm opt}')
\leq (1+3\epsilon)w^-(S_{\rm opt})+w_{\max}/t$.
Now, we obtain
\begin{eqnarray*}
w(S_2) 
& \leq & t\cdot w^-(S_2)+t\cdot n\\
& \leq & t\cdot w^-(S_{\rm opt}')+t\cdot n
\,\,\,\,\,\,\,\,\,\,\,\,\,\,\,\,\,\,\mbox{(optimality of $S_2$ w.r.t. $w^-$)}\\
& \leq & (1+3\epsilon)\cdot t\cdot w^-(S_{\rm opt})+w_{\max}+\epsilon^2 w(S_1)\\
& \leq & (1+3\epsilon)\cdot w(S_{\rm opt})+w_{\max}+2\epsilon^2 w(S_{\rm opt})\\
& = & (1+3\epsilon+2\epsilon^2)\cdot w(S_{\rm opt})+w_{\max},
\end{eqnarray*}
which completes the proof.
\end{proof}
A simple modification of the hardness proof in \cite{bafulemamasatu,bafulemamasatu2} 
allows to show our hardness result.

\begin{proof}[Proof of Theorem \ref{theoremhard}]
It follows easily from the standard hardness proofs (cf. Corollary 15.27 in \cite{kovy})
that {\sc Subset Sum} remains NP-complete for instances
$(c_1,\ldots,c_n,K)$ of positive integers 
with 
\begin{itemize}
\item $\max\{ c_i:i\in [n]\}<K<c_1+\ldots+c_n$, and
\item $2\min\{ c_i:i\in [n]\}\geq \max\{ c_i:i\in [n]\}$.
\end{itemize}
In order to ensure the crucial second condition, 
one can first proceed as in the proof of Corollary 15.27 in \cite{kovy},
and then increase each $c_j$ by $(n+1)^{3m+1}$ and
$K$ by $n\cdot (n+1)^{3m+1}$.

Let $(c_1,\ldots,c_n,K)$ be an instance of {\sc Subset Sum} satisfying the above conditions.
Let $T$ be a star of order $n+2$,
and let $w:V(T)\to \mathbb{N}$
be such that the center vertex receives weight $1$,
and the remaining vertices receive weights $c_1,\ldots,c_n,K+1$.
Since $\max\{ c_i:i\in [n]\}<K+1<c_1+\ldots+c_n+1$, every connected safe set in $T$ contains the center vertex.
Since some vertex in $T$ has weight $K+1$, we have $cs(T,w)\geq K+1$.
Since the center vertex together with the vertex of weight $K+1$ is a connected safe set, we have $cs(T,w)\leq K+2$.
Furthermore, $cs(T,w)=K+1$ if and only if 
there is a subset $S$ of $[n]$ such that $1+\sum_{i\in S}c_i=K+1$,
which happens if and only if $(c_1,\ldots,c_n,K)$ is a yes-instance of {\sc Subset Sum}.

If $u^*$ is the center vertex, then 
\begin{eqnarray*}
4\min\Big\{ w(u):u\in V(T)\setminus \{ u^*\}\Big\} & = & 4\min\{ c_i:i\in [n]\}\\
&\geq &\frac{2(c_1+\ldots+c_n)}{n}\\
&\geq &\frac{c_1+\ldots+c_n+K}{n}\\
&\geq & \frac{c_1+\ldots+c_n+K+2}{n+2}\\
&=&\frac{w(T)}{n(T)},
\end{eqnarray*}
which completes the proof.
\end{proof}
We proceed to the proof of our bound.

\begin{proof}[Proof of Theorem \ref{theorembound1}]
If $\lceil \omega/2\rceil\geq \lceil n/3\rceil+1$, 
then one block $B$ of $G$ is a clique of order $\omega\geq 2\lceil n/3\rceil+1$ 
that contains at most $n-n(B)\leq n/3-1\leq \lceil \omega/2\rceil$ cutvertices of $G$.
If $S$ is a subset of $V(B)$ of order $\lceil \omega/2\rceil$ 
that contains all cutvertices of $G$ in $B$,
then one component of $G-S$ is $B-S$, which has order $\lfloor \omega/2\rfloor$, while all remaining components of $G-S$ altogether have order $n-n(B)\leq \lceil \omega/2\rceil$,
that is, $S$ is a connected safe set, and the statement follows.
Now, we may assume that $\lceil \omega/2\rceil\leq \lceil n/3\rceil$.

Let $S$ be a set of $\lceil n/3\rceil$ vertices of $G$ chosen in such a way that a largest component $D$ of $G-S$ is smallest possible.
For a contradiction, suppose that $n(D)\geq \lceil n/3\rceil+1$.
Note that $n-|S|-n(D)\leq n/3-1<|S|<n(D)$, 
which implies that $D$ is the unique component of $G-S$ 
that contains more vertices than $S$.

First, suppose that $S$ is not contained in only one block of $G$.
Let $B$ be a block of $G$ that contains a vertex of $S$ and a vertex $u$ of $D$. If $v$ is a vertex in $S\setminus V(B)$ that is not a cutvertex of $G[S]$, $S'=(S\setminus \{ v\})\cup \{ u\}$, and $D'$ is a largest component of $G-S'$, then either $V(D')$ is a proper subset of $V(D)$, or $V(D')$ is a subset of the set $V(G)\setminus (S\setminus \{ v\}\cup V(D))$ of order at most $n-(\lceil n/3\rceil-1)-(\lceil n/3\rceil+1)\leq n/3$.
In both cases, $n(D')<n(D)$, contradicting the choice of $S$.
Note that the vertex $v$ can be chosen as a vertex in $S$ at maximum distance from $S\cap V(B)$.

Next, suppose that $S$ is contained in only one block $B$ of $G$.
If $D$ does not contain a vertex of $B$, then some cutvertex $x$ of $G$ in $S$ has a neighbor $u$ in $D$. If $v$ is a vertex in $S\setminus \{ x\}$, $S'=(S\setminus \{ v\})\cup \{ u\}$, and $D'$ is a largest component of $G-S'$, then 
either $V(D')$ is a proper subset of $V(D)$
or $V(D')$ is a subset of $V(G)\setminus (S\setminus \{ v\}\cup V(D))$
of order at most $n/3$.
Again, in both cases $n(D')<n(D)$, contradicting the choice of $S$.
If $D$ intersects $B$, then $|S|+n(D)>\omega\geq n(B)$ implies that $D$ contains a cutvertex $u$ of $G$ in $B$.
Since $n-|S|-n(D)<|S|$, 
the set $S$ contains a vertex $v$ of $B$ that is not a cutvertex of $G$.
If $S'=(S\setminus \{ v\})\cup \{ u\}$, 
and $D'$ is a largest component of $G-S'$, 
then $D'$ has strictly less vertices than $D$,
contradicting the choice of $S$.
\end{proof}
Fujita et al.~\cite{fumasa} show that $s(T)\geq n/(k+1)$ for every tree $T$ of order $n\geq 2$ with at most $k\geq 2$ leaves.
Adapting their proof easily implies that 
$s(G)\geq n/(k+1)$ for every block graph $G$ of order $n\geq 2$ with at most $k\geq 2$ endblocks.

\end{document}